\documentclass[12pt]{amsart}

\usepackage{amsmath}
\usepackage{amsfonts}
\usepackage{amssymb}
\usepackage{amsthm}

\usepackage{upgreek}
\usepackage{bigdelim}
\usepackage{multirow}
\usepackage{hhline}
\usepackage{enumerate}
\usepackage{comment}
\usepackage{amsthm,verbatim}
\usepackage{latexsym}
\usepackage{fancyhdr}

\newcommand{\R}{\mathbb{R}}

\renewcommand{\epsilon}{\varepsilon}
\renewcommand{\Lambda}{\Uplambda}

\newtheorem{theorem}{Theorem}[section]
\newtheorem{lemma}[theorem]{Lemma}

\newtheorem{corollary}[theorem]{Corollary}

\begin{document}


\baselineskip=17pt




\title[Infinite Transformations  ] {Weak Rational Ergodicity Does Not Imply Rational Ergodicity}

\author[T. M. Adams]{Terrence M. Adams}

\address{U.S. Government\\9161 Sterling Dr.\\Laurel, MD 20723}

\email{tmadam2@tycho.ncsc.mil}

\email{ }

\author[C. E. Silva]{Cesar E. Silva}

\address{Department of Mathematics and Statistics, Williams College, 
\\ Williamstown, MA 01267} 

\email{csilva@williams.edu}


\begin{abstract}
We extend the notion of rational ergodicity to $\beta$-rational ergodicity for $\beta > 1$. 
Given $\beta \in \R$ such that $\beta > 1$, we construct an uncountable family of rank-one infinite measure preserving 
transformations that are weakly rationally ergodic, but are not $\beta$-rationally ergodic. 
The established notion of rational ergodicity corresponds to 2-rational ergodicity. 
Thus, this paper answers an open question by showing that weak rational ergodicity 
does not imply rational ergodicity. 
\end{abstract}



\maketitle 


\section{Introduction}

In this paper we consider ergodic properties of  invertible, infinite measure-preserving transformations 
 on $\sigma$-finite, nonatomic, Lebesgue measure spaces.  As is well known, the averages in the ergodic theorem, for ergodic infinite measure-preserving transformations, converge to $0$. 
 In 1977,   Aaronson \cite{Aa77} introduced the notion of weak rational ergodicity, where an  ergodic average for a certain class of sets converges to the expected limit, similar to the case of finite invariant measure. Aaronson also defined in the same article the notion of rational ergodicity and proved that rational ergodicity implies weak rational ergodicity but left the question of equivalence open.  In this paper we define for each real number $\beta>1$ a notion  of $\beta$-rational ergodicity, with $2$-rational ergodicity agreeing with the usual rational ergodicity. We then construct examples, for each $\beta>1$, of $\beta$-rationally ergodic transformations which are not weakly rationally ergodic. Thus in particular we show that weak rational ergodicity does not imply rational ergodicity for infinite measure-preserving transformations.

Let  $\beta $ be a real number and assume   that $\beta > 1$. 
A transformation $T$ is said to be {\bf $\beta$-rationally ergodic} if it is conservative ergodic and there exists a set $F$ 
of positive finite measure such that 
\[
\liminf_{n\to \infty} \frac{ ( \int_F \sum_{i=0}^{n-1} 
I_F(T^ix) d\mu )^{\beta} }{\int_F (\sum_{i=0}^{n-1} I_F(T^ix))^{\beta} d\mu } 
> 0 .
\]
The  notion of rational ergodicity  in \cite{Aa77} corresponds 
to $2$-rational ergodicity.  A direct application of H\"{o}lder's inequality 
shows that if $\beta_2 > \beta_1 > 1$, and $T$ is $\beta_2$-rationally ergodic, 
then $T$ is $\beta_1$-rationally ergodic. Furthermore, $T$ is said to be  {\bf weakly rationally ergodic}  \cite{Aa77} if
it is conservative ergodic and there exists  a set $F$ 
of positive finite measure such that, if we set $a_n(F)=\sum_{k=0}^{n-1}\mu(F\cap T^k F)/\mu(F)^2$,
then
\[
\lim_{n\to\infty}\frac{1}{a_n(F)}\ \sum_{k=0}^{n-1}\mu(A\cap T^k B)=\mu(A)\mu(B),
\] for all measurable $A,B\subset F$.

\section{Construction of the Examples}

Let $k_n, \ell_n$, and $m_n$ be sequences of natural numbers. 

\subsection{Initialization}
Let $I_0$ be an interval with positive length. 
Cut $C_0 = I_0$ into $k_0$ subintervals of equal length. 
Label the subintervals $C_0(i)$ for $0\leq i < k_0$. 
Stack $\ell_0$ subintervals on top of $C_0(i)$ for 
$0\leq i < k_0-1$ to form $k_0-1$ subcolumns 
of height $\ell_0 +1$. Label these subcolumns $\bar{C}_0(i)$ 
for $0\leq i < k_0-1$. Stack the subcolumns $\bar{C}_0(i)$ for $0\leq i < k_0 - 1$ 
from left to right to form a single subcolumn $\bar{C}_0$ 
of height $(k_0-1)(\ell_0 + 1)$. 
Let $\bar{C}_0(k_0-1) = C_0(k_0-1)$, which is a subcolumn of height 1. 
We have that both bases of towers $\bar{C}_0$ and $C_0(k_0 - 1)$ have 
the same measure: 
$$
\mu(C_0(k_0 - 1)) = \mu(C_0(0)) = \frac{1}{k_0} \mu (I_0) . 
$$
Cut each subcolumn $\bar{C}_0$ and $C_0(k_0 - 1)$ into $m_0$ subcolumns 
and stack from left to right. In particular, 
let $C_0(i,j)$ be the $j^{th}$ subcolumn of $\bar{C}_0(i)$ for $0\leq j < m_0$. 
Thus, as measurable sets, 
\[
\bar{C}_0 = \bigcup_{i=0}^{k_0 - 2} \bigcup_{j=0}^{m_0 - 1} C_0(i,j) 
\]
and 
\[
\bar{C}_0(k_0 - 1) = C_0(k_0 - 1) = \bigcup_{j=0}^{m_0 - 1} C_0(k_0 - 1, j) .  
\]
Stack the $C_0(k_0-1)$ subcolumn of width ${1} / {m_0 k_0}$ 
on top of the $\bar{C}_0$ subcolumn of the same width to form a single column of height 
$m_0(k_0-1)(\ell_0 + 1) + m_0$. Place $m_0(k_0-1)(\ell_0 + 1) + m_0$ spacers 
on top to form column $C_1$ of height 
$$h_1 = 2m_0(k_0-1)(\ell_0 + 1) + 2m_0 . $$

\subsection{General Step}
Let $C_n$ be a column of height $h_n$. 
Use the same procedure as above to cut $C_n$ into $k_{n}$ 
subcolumns of equal width. Separate the subcolumns into the 
first $k_{n}-1$ subcolumns and the last subcolumn. 
Add $\ell_{n}$ subintervals on top of the first $k_{n}-1$ 
subcolumns, and then stack from left to right to form 
a single subcolumn of height $(h_n + \ell_{n}) (k_{n} - 1)$. 
For the last subcolumn of height $h_n$, cut into $m_{n}$ 
subcolumns of equal width and stack from left to right. 
Also, cut the first subcolumn of height $(h_n + \ell_{n}) (k_{n} - 1)$ 
into $m_{n}$ subcolumns of equal width and stack from left 
to right. This produces two subcolumns of equal width. 
Stack the shorter subcolumn on top of the taller subcolumn, 
and add an equal number of spacers to form a single column $C_{n+1}$ 
of height: 
$$
h_{n+1} = 2m_{n}(h_n + \ell_{n}) (k_{n} - 1) + 2m_{n}h_n .
$$
Also, set $H_n = h_n + \ell_n$. 

As in the initialization, let $C_n(i)$ be the $i^{th}$ subcolumn 
from cutting $C_n$ into $k_n$ subcolumns of equal width. 
Let $\bar{C}_n(i)$ be the $i^{th}$ subcolumn including 
the $\ell_n$ spacers added on top of $C_n(i)$ for $0\leq i < k_n - 1$. 
Set $\bar{C}_n(k_n - 1) = C_n(k_n - 1)$. 
Finally, let $C_n(i,j)$ be the $j^{th}$ subcolumn of $\bar{C}_n(i)$ 
for $0\leq j < m_n$. 
For a given sequence $v=(v_n)=(k_n,\ell_n,m_n)$, this procedure 
produces a $\sigma$-finite measure preserving transformation 
$T_v:X\to X$ where $X=\bigcup_{n=1}^{\infty} C_n$. 

Suppose $L$ is the union of all subintervals added throughout the construction. 
Then $X\setminus L = I_{0,0}$ has finite $\mu$ measure, and the induced transformation 
$(T_v)_{X\setminus L}$ is ergodic and rank-one. 
For convenience, set $\mu(I_{0,0}) = 1$ and 
let $\hat{T}_v$ denote the probability preserving invertible transformation 
obtained by inducing $T_v$ on the set $X\setminus L$. 
Also, let $\hat{h}_n$ be the tower height of the rank-one transformation $\hat{T}_v$ corresponding to the tower of height $h_n$ for $T_v$. 

\subsection{$\alpha$-family}
Given a real number $x$, let 
$\lfloor x \rfloor = \max{ \{ \ell \in \mathbb N : \ell \leq x \} }$. 
In this section, we restrict $v=(k_n,\ell_n,m_n)$ such that the collection 
of transformations $T_v$ gives a sufficiently rich class of counterexamples. 
Let $\alpha \in \R$ be such that $0 < \alpha < 1$. 
Define the class $V_{\alpha}$ of infinite measure preserving transformations 
such that 
\[
V_{\alpha} = \{ T_v : v=(n+1,\lfloor n^{\alpha}\rfloor h_n,m_n), 
\lim_{n\to \infty}\frac{\lfloor n^{\alpha} \rfloor}{m_n}=0 \} . 
\]
Define the collection 
\[
V = \bigcup_{0<\alpha < 1} V_{\alpha} . \
\]
For $n\in \mathbb N$, $C_n(k_n - 1)$ is the last subcolumn of $C_n$. 
It is cut into $m_n$ subcolumns of equal width, and labeled $C_n(k_n - 1, j)$ 
for $0\leq j < m_n$. 
Define 
\[
D_n = \bigcup_{j=\lfloor n^{\alpha} \rfloor}^{m_n - 1} C_n(k_n - 1, j) . 
\]

\section{Main Results}
In this section, we state our main results, and give 
the proofs in the following two sections. 
The collection $V$ provides all the necessary counterexamples, 
including a solution to the question 
of whether weak rational ergodicity implies rational ergodicity. 

\begin{theorem}
\label{WeakRatErg}
Each transformation $T\in V$ is a weakly  rationally ergodic 
infinite measure preserving transformation. 
\end{theorem}

\begin{theorem}
\label{RatErg}
Suppose $\alpha, \beta \in \R$ such that $0 < \alpha < 1$ and $\alpha \beta > 1$. 
If $T\in V_{\alpha}$, then for every set $F$ of positive finite measure, 
\[
\liminf_{n\to \infty} \frac{ ( \int_F \sum_{i=0}^{H_n-1} I_F(T^ix) d\mu )^{\beta} }{\int_F (\sum_{i=0}^{H_n-1} I_F(T^ix))^{\beta} d\mu } 
= 0 .
\]
In other words, $T$ is not $\beta$-rationally ergodic. 
\end{theorem}

\begin{corollary}
\label{RatErgCor1}
For each $\beta > 1$, there exists an infinite measure preserving transformation 
$T$ such that $T$ is weakly rationally ergodic, but not $\beta$-rationally ergodic. 
\end{corollary}

\begin{proof} 
Given $\beta > 1$, choose $\alpha < 1$ such that $\alpha \beta > 1$. 
Let $T$ be any transformation in $V_{\alpha} \subset V$. 
By Theorem \ref{WeakRatErg}, $T$ is weakly rationally ergodic, and 
by Theorem \ref{RatErg}, $T$ is not $\beta$-rationally ergodic. 
\end{proof}

\begin{corollary}
There exist infinite measure preserving transformations $T$ 
that are weakly rationally ergodic, but are not rationally ergodic. 
\end{corollary}

\begin{proof} 
Apply Corollary \ref{RatErgCor1} with $\beta = 2$. 
\end{proof}

\section{Weakly Rationally Ergodic Examples}
To establish weak rational ergodicity, we set $F = I_0$. 
Given $N\in \mathbb N$, define 
\[
a_N(\alpha) = \sum_{i=0}^{N-1} \mu (F\cap T_{\alpha}^i F) . 
\]
Let $i,n\in \mathbb N$ be such that $0 \leq i \leq n$, and 
$F_n(i) = F \cap C_n(i)$.  Define 
\[
b_N^n(\alpha) =  \sum_{i=0}^{N-1} [ \mu (F\cap T_{\alpha}^i F_n(n))  + \mu (F_n(n)\cap T_{\alpha}^i F) ] .
\]

\begin{lemma}
\label{lem1}
Suppose $t_n \in \mathbb N$ such that $0 < t_n < h_{n+1}$ for $n\in \mathbb N$. 
For $\alpha \in (0,1)$, 
\[
\lim_{n\to \infty} \frac{b_{t_n}^n(\alpha)}{a_{t_n}(\alpha)} = 0 . 
\]
\end{lemma}

\begin{proof} 
Let $T \in V_{\alpha}$, and $F_n(i,j) = F \cap C_n(i,j)$ for $n\in \mathbb N$. 
First, suppose $t_n < m_n h_n$. Let $p_n = \lfloor \frac{(n-1)m_n}{n} \rfloor$ and let 
\[
G_n = \bigcup_{i=0}^{k_n - 2} \bigcup_{j=0}^{p_n} F_n(i,j) . 
\]
Suppose $r\in \mathbb N$ such that $0\leq r < m_n h_n - H_n$. 
Then for $0\leq i < n$ and $0\leq j < p_n$, 
\[
\sum_{t=0}^{H_n-1} \mu(G_n \cap T^{t+r} F_n(i,j)) = \sum_{t=0}^{h_n-1} \mu(F \cap T^{t+r} F_n(n,0)) . 
\]
Also, for $r\in \mathbb N$ such that $h_n \leq r < H_n$ and $n$ sufficiently large, 
\[
\sum_{t=0}^{r-1} \mu(G_n \cap T^{t} F_n(i,j)) > \frac{1}{3} \sum_{t=0}^{h_n-1} \mu(F \cap T^{t} F_n(n,0)) . 
\]
Thus, for $n$ sufficiently large, 
\begin{eqnarray*}
\sum_{t=0}^{r-1} \mu(G_n \cap T^{t} G_n) &>& \frac{p_n(k_n-1)}{3} \sum_{t=0}^{h_n-1} \mu(F \cap T^{t} F_n(n,0)) \\ 
&>&  \frac{m_nn}{6} \sum_{t=0}^{h_n-1} \mu(F \cap T^{t} F_n(n,0)) \\ 
&\geq&  \frac{n}{6} \sum_{t=0}^{h_n-1} \mu(F \cap T^{t} F_n(n)) \\ 
&=&  \frac{n}{6(n^{\alpha}+1)} (n^{\alpha} + 1) \sum_{t=0}^{h_n-1} \mu(F \cap T^{t} F_n(n)) \\ 
&\geq&  \frac{n}{6(n^{\alpha}+1)} \sum_{t=0}^{r-1} \mu(F \cap T^{t} F_n(n)) . 
\end{eqnarray*}
Since 
\[
\lim_{n\to \infty} \frac{6(n^{\alpha} + 1)}{n} = 0 , 
\]
then our lemma holds for $h_n \leq t_n < H_n$. Similarly, it holds for $0 < t_n < H_n$. 
To establish for $H_n \leq t_n < m_n h_n$, 
let $t_n = q_n H_n + r_n$ where $q_n \in \mathbb N$ and $0\leq r_n < H_n$. 
Then 
\begin{eqnarray}
\sum_{t=0}^{t_n-1} \mu(G_n \cap T^{t} G_n) &=& 
\sum_{t=0}^{r_n-1} \mu(G_n \cap T^{t} G_n) \label{first} \\ 
&+& \sum_{q=0}^{q_n-1} \sum_{t=0}^{H_n-1} \mu(G_n \cap T^{t+qH_n + r_n} G_n) \label{second}
\end{eqnarray} 
We already established our lemma for (\ref{first}), so we now handle (\ref{second}). 
\begin{align*}
\sum_{q=0}^{q_n-1} \sum_{t=0}^{H_n-1} \mu &(G_n \cap T^{t+qH_n + r_n} G_n) \\ 
&= \sum_{q=0}^{q_n-1} \sum_{t=0}^{H_n-1} \sum_{i=0}^{k_n-2} \sum_{j=0}^{p_n - 1} \mu(G_n \cap T^{t+qH_n+r_n} F_n(i,j)) \\ 
&= \sum_{q=0}^{q_n-1} \sum_{t=0}^{h_n-1} \sum_{i=0}^{n-1} \sum_{j=0}^{p_n - 1} \mu(F \cap T^{t+qh_n+r_n} F_n(n,0)) \\ 
&= \sum_{q=0}^{q_n-1} \sum_{t=0}^{h_n-1} n p_n \mu(F \cap T^{t+qh_n+r_n} F_n(n,0)) \\ 
&\geq \sum_{q=0}^{q_n-1} \sum_{t=0}^{h_n-1} 
\frac{n m_n}{2(n^{\alpha}+1)} (n^{\alpha}+1)\mu(F \cap T^{t+qh_n+r_n} F_n(n,0)) \\ 
&\geq \sum_{q=0}^{q_n-1} \sum_{t=0}^{H_n-1} 
\frac{n}{2(n^{\alpha}+1)} \mu(F \cap T^{t+qh_n+r_n} F_n(n)) \\ 
&= \sum_{t=r_n}^{t_n-1} \frac{n}{2(n^{\alpha}+1)} \mu(F \cap T^{t} F_n(n))
\end{align*}
Once again, since 
\[
\lim_{n\to \infty} \frac{2(n^{\alpha} + 1)}{n} = 0 , 
\]
then our lemma is established for $0 < t_n < m_n h_n$. 

Note that $\mu(F\cap T^t F_n(n)) = 0$ for $m_nh_n \leq t < h_{n+1} - m_n h_n$. 
If $t_n \geq h_{n+1} - m_n h_n$, the partial sum 
\[
\sum_{t=h_{n+1} - m_n h_n}^{t_n-1} \mu (F\cap T^t F_n(n)) 
\]
may be handled in a similar manner as above. 
Also, the case of $\sum_{t=0}^{t_n-1} \mu (T^tF\cap F_n(n))$ 
follows in a similar way. This completes the proof of our lemma. 

\end{proof}

\begin{lemma} 
\label{case0}
Let $T\in V_{\alpha}$ such that $0 < \alpha < 1$. 
Also, let $F=I_0$ and $A,B \subset F$ be measurable. 
Suppose $t_n = q_n H_n$ such that $1\leq t_n < h_{n+1}$ for $n\in \mathbb N$. 
If 
\[
a_{t_n} = \hat{h}_n q_n (1 - \frac{q_n}{2(n+1)m_n}) , 
\]
then 
\[
\lim_{n\to \infty} \frac{1}{a_{t_n}} \sum_{i=0}^{t_n-1} \mu(A\cap T^i B) = \mu(A)\mu(B) .
\]
\end{lemma}
\begin{proof}
This lemma may be proven using a counting argument on the 
$k_n m_n = (n+1) m_n$ subcolumns comprising $C_n$. 
By Lemma \ref{lem1}, we may assume $A\cap C_n(n) = \emptyset$ and 
$B\cap C_n(n) = \emptyset$. 
In this case, we may disregard $i \geq (n+1)m_n H_n$ in the summation, 
since $\mu (A\cap T^iB) = 0$ for $h_{n+1} > i \geq (n+1)m_n H_n$. 
In the above summation, each of the first $(nm_n - q_n)$ subcolumns, 
produces on average approximately weight 
\[
\frac{q_n \hat{h}_n \mu (A) \mu (B)}{(n+1)m_n} . 
\]
The next $(q_n - 1)$ subcolumns produces an approximate total weight 
\[
\frac{(q_n^2 - q_n)}{2(n+1)m_n} \hat{h}_n \mu (A) \mu (B) . 
\]
Therefore, the total weight is approximately 
\begin{align*}
\hat{h}_n q_n \mu (A)\mu (B) & \frac{(2nm_n - 2q_n + q_n - 1)}{2(n+1)m_n} \\ 
\sim & \ \hat{h}_n q_n \mu (A)\mu (B) ( 1 - \frac{q_n}{2(n+1)m_n} ) . 
\end{align*}

\end{proof}

\noindent 
The previous lemma gives a formula for $a_t$ for certain values 
of $t\in \mathbb N$. 
Here we show how to define $a_t$ for all $t$ sufficiently large. 
Given $t\in \mathbb N$, choose $n\in \mathbb N$ such that 
$h_n \leq t < h_{n+1}$. 
Write $t = qH_n + r$ such that $0\leq r < H_n$. 
To obtain the value of $a_t$, we separate into three cases 
based on the value of $r$: 
\begin{enumerate}
\item $h_n \leq r < H_n - h_n$, \label{case1} 
\item $r < h_n$, \label{case2} 
\item $r \geq H_n - h_n$. \label{case3} 
\end{enumerate}
{\bf Case \ref{case1}:} 
Define $a_t$ as 
\[
a_t = q \hat{h}_n (1 - \frac{q}{2(n+1)m_n}) + \frac{1}{2} \hat{h}_n .
\]

\noindent 
{\bf Case \ref{case2}:} 
Let $r = q^{\prime} H_{n-1} + r^{\prime}$ where $0\leq r^{\prime} < H_{n-1}$. 
Define $a_t$ as 
\[
a_t = q \hat{h}_n (1 - \frac{q}{2(n+1)m_n}) + q^{\prime} \hat{h}_{n-1} (1 - \frac{q^{\prime}}{2nm_{n-1}}) . 
\]

\noindent 
{\bf Case \ref{case3}:} 
Let $H_n - r = q^{\prime\prime} H_{n-1} + r^{\prime\prime}$ where $0\leq r^{\prime\prime} < H_{n-1}$. 
Define $a_t$ as 
\[
a_t = (q + 1) \hat{h}_n (1 - \frac{q}{2(n+1)m_n}) - q^{\prime\prime} \hat{h}_{n-1} (1 - \frac{q^{\prime\prime}}{2nm_{n-1}}) (1 - \frac{q}{(n+1)m_{n}}) . 
\]

\begin{theorem}
\label{mainWeakErgThm}
Fix $\alpha \in (0,1)$. 
Let $T\in V_{\alpha}$, $F=I_0$ and $A,B \subset F$ be measurable. 
Suppose $a_t$ is defined as above for $t\in \mathbb N$.  Then 
\[
\lim_{t\to \infty} \frac{1}{a_{t}} \sum_{i=0}^{t-1} \mu(A\cap T^i B) = \mu(A)\mu(B) .
\]
\end{theorem}
\begin{proof}
By passing to a subsequence, we may assume each of 
$q, r, q^{\prime},r^{\prime}, q^{\prime \prime}, r^{\prime \prime}$ 
tends to $\infty$ or is bounded. 
For case \ref{case1}, 
separate $a_t = b_t + c_t$ where 
$b_t = q \hat{h}_n (1 - \frac{q}{2(n+1)m_n})$ and $c_t = \frac{1}{2} \hat{h}_n$. 
Thus, 
\begin{align} 
\frac{1}{a_t}\sum_{i=0}^{t-1} \mu(A\cap T^iB)  &= \label{sum1-1} \\ 
\frac{b_t}{a_t} & \frac{1}{b_t} \sum_{i=0}^{qH_n - 1} \mu (A\cap T^iB) +  
\frac{c_t}{a_t} \frac{1}{c_t} \sum_{i=0}^{r - 1} \mu (A\cap T^{qH_n+i}B) \label{sum1-2}
\end{align} 
If $q=q(t) \to \infty$ as $t \to \infty$, then 
${c_t} / {a_t} \to 0$ as $t\to \infty$, and 
we can disregard the second half of (\ref{sum1-2}). 
In this case, our theorem follows from applying Lemma \ref{case0} 
to the first half of (\ref{sum1-2}). 
Otherwise, the first half of (\ref{sum1-2}) is approximated by 
\[
\frac{b_t}{a_t} \mu (A)\mu (B) . 
\]
For case \ref{case1},
most blocks of height $h_n$ move forward into the spacers added to $C_n$ under $T^{r}$. 
Since the blocks do not return to its neighboring block due to the spacers, then we get half 
of the intersection that would occur under $\hat{T}^{\hat{h}_n}$. 
Note, due to symmetry, 
\[
\sum_{i=0}^{r-1} \mu (A\cap T^{qH_n-i}B) \sim \sum_{i=0}^{r-1} \mu (A\cap T^{qH_n+i}B) . 
\]
Thus, the second half of (\ref{sum1-2}) is approximated by 
\[
\frac{c_t}{a_t} \mu (A)\mu (B) . 
\]
Hence, for case \ref{case1}, 
\[
\lim_{t\to \infty} \frac{1}{a_{t}} \sum_{i=0}^{t-1} \mu(A\cap T^i B) = \mu(A)\mu(B) .
\]
For case (\ref{case2}), if $r=0$, then our theorem holds by Lemma \ref{case0}. 
Likewise, if $r$ is bounded, our theorem holds again by Lemma \ref{case0}. 
If $c_t = q^{\prime} \hat{h}_{n-1} (1 - \frac{q^{\prime}}{2nm_{n-1}})$, 
and $r = q^{\prime} H_{n-1}^{\prime} + r^{\prime}$, then $q^{\prime}$ bounded 
implies ${c_t} / {a_t} \to 0$. Otherwise, 
\begin{align} 
\frac{1}{c_t} \sum_{i=0}^{r - 1} & \mu (A\cap T^{i}T^{qH_n}B) \\ 
=& \frac{1}{c_t} \sum_{i=0}^{q^{\prime}H_{n-1} - 1} \mu (A\cap T^{i} T^{qH_n}B) + 
 \frac{1}{c_t} \sum_{i=0}^{r^{\prime} - 1} 
\mu (A\cap T^{i}T^{q^{\prime} H_n^{\prime}} T^{qH_n}B) . \label{sum2-2} 
\end{align} 
By the previous argument, we can disregard the second half of (\ref{sum2-2}), 
and hence, 
\begin{align} 
\lim_{t\to \infty} \frac{1}{c_t} \sum_{i=0}^{r - 1} & \mu (A\cap T^{i}T^{qH_n}B) 
= \mu (A) \mu (B) . 
\end{align} 
For case \ref{case3}, let 
\begin{align*}
\frac{1}{a_t} & \sum_{i=0}^{t - 1} \mu (A\cap T^{i}B) = 
\frac{1}{(b_t - c_t)} \sum_{i=0}^{t - 1} \mu (A\cap T^{i}B) = \\ 
& \frac{b_t}{(b_t - c_t)} \frac{1}{b_t} \sum_{i=0}^{(q+1)H_n - 1} \mu (A\cap T^{i}T^{qH_n}B)
- \frac{c_t}{(b_t - c_t)} \frac{1}{c_t} \sum_{i=r}^{H_n - 1} \mu (A\cap T^{qH_n + i}B)
\end{align*}
where 
$b_t = (q + 1) \hat{h}_n (1 - \frac{q}{2(n+1)m_n})$ and 
$c_t = q^{\prime\prime} \hat{h}_{n-1} (1 - \frac{q^{\prime\prime}}{2nm_{n-1}} )(1 - \frac{q}{(n+1)m_{n}})$ . 
If $q \to \infty$ as $t \to \infty$, then ${c_t} / {b_t} \to 0$ as $t \to \infty$ and our result follows. 
Otherwise, 
\begin{align*}
\lim_{t\to \infty} \frac{1}{c_t} \sum_{i=r}^{H_n - 1} \mu (A\cap T^{qH_n + i}B) 
=& \lim_{t\to \infty} \frac{1}{c_t} \sum_{i=0}^{H_n - r - 1} \mu (A\cap T^{(q+1)H_n - i - 1}B) \\ 
=& \mu (A) \mu (B) 
\end{align*}
and our proof is complete. 
\end{proof} 

\noindent
By setting $a_t(F)=a_t$, 
Theorem \ref{mainWeakErgThm} clearly implies Theorem \ref{WeakRatErg}. 
Therefore, we have established that 
each $T\in V$ is weakly rationally ergodic.  

\section{Non-Rationally Ergodic Examples} 
Suppose $\alpha, \beta \in \R$ such that $0 < \alpha < 1$ and 
$\alpha \beta > 1$.  In this section, we prove for each $T \in V_{\alpha}$, 
$T$ is not $\beta$-rationally ergodic.  We note that there are many examples 
that have been shown to be rationally ergodic, see e.g. \cite{Aa97}. In particular,
rank-one transformations with bounded cuts have been shown to be rationally ergodic \cite{DGPS14}.
See also \cite{AKW13,BSSSW}.
Maharam transformations are not weakly rationally ergodic \cite{Aa77}, though they are not rank-one \cite{BSSSW}.

Before we prove the main theorem, we state and prove the following 
basic lemma. 

\begin{lemma}
\label{smallintegral}
Let $T$ be an invertible infinite measure preserving ergodic transformation. Suppose 
for each set $F$ of positive finite measure, there exists a sequence $t_n \in \mathbb N$ 
and $F_n \subset F$ of positive measure such that 
$\mu(F_n) \to 0$ as $n\to \infty$ and 
\[
\limsup_{n\to \infty} \frac{\int_{F_n} \sum_{i=0}^{t_n-1} I_{F} (T^i x) d\mu}
{\int_{F} \sum_{i=0}^{t_n-1} I_{F} (T^i x) d\mu} > 0 . 
\]
Then $T$ is not $\beta$-rationally ergodic for each $\beta > 1$. 
\end{lemma}

\begin{proof} 
Let $\beta > 1$ and $\gamma = \frac{\beta}{\beta - 1}$. 
Without loss of generality, by passing to a subsequence, assume there exist $\eta > 0$ such that 
for all $n\in \mathbb N$, 
\[
\frac{\int_{F_n} \sum_{i=0}^{t_n-1} I_{F} (T^i x) d\mu}
{\int_{F} \sum_{i=0}^{t_n-1} I_{F} (T^i x) d\mu} > \eta . 
\]
By H\"{o}lder's inequality, 
\begin{eqnarray*}
\int_{F_n} ( \sum_{i=0}^{t_n-1} I_{F}(T^i x) ) I_{F_n}(x) d\mu &\leq& 
[\int_{F_n} (\sum_{i=0}^{t_n-1} I_{F}(T^i x) )^{\beta} d\mu ]^{{1} / {\beta}} 
\mu(F_n)^{{1} / {\gamma}} 
\end{eqnarray*} 
Thus, 
\begin{eqnarray*}
\frac{ [\int_{F_n} \sum_{i=0}^{t_n-1} I_{F}(T^i x) d\mu]^{\beta} }
{ \int_{F_n} (\sum_{i=0}^{t_n-1} I_{F}(T^i x) )^{\beta} d\mu } 
&\leq& \mu(F_n)^{{\beta} / {\gamma}} 
\end{eqnarray*} 
Therefore, 
\begin{eqnarray*}
\frac{ [\int_{F} \sum_{i=0}^{t_n-1} I_{F}(T^i x) d\mu]^{\beta} }
{ \int_{F} (\sum_{i=0}^{t_n-1} I_{F}(T^i x) )^{\beta} d\mu }  
&<& (\frac{1}{\eta})^{\beta} 
\frac{ [\int_{F_n} ( \sum_{i=0}^{t_n-1} I_{F}(T^i x) ) d\mu]^{\beta} }
{ \int_{F_n} (\sum_{i=0}^{t_n-1} I_{F}(T^i x) )^{\beta} d\mu }  \\ 
&\leq& (\frac{1}{\eta})^{\beta} \mu(F_n)^{{\beta} / {\gamma}}  \to 0 
\end{eqnarray*} 
as $n \to \infty$. 
\end{proof} 

\noindent 
We will use the following lemma from \cite{AdSi14}; we include the proof for completeness. 

\begin{lemma} (Mixing Lemma) 
\label{mixlem}
Let $(X,\gamma)$ be a probability space.  
Let $E_i\subset{X}$ be a sequence of pairwise independent sets satisfying 
$$\sum_{i=1}^{\infty}\gamma (E_i)=\infty .$$ 
Given any measurable set $E\subset X$ and $\varepsilon >0$, there exist 
infinitely many positive integers $i$ such that 
$\gamma (E\cap E_i)>(\gamma (E)-\varepsilon )\gamma (E_i)$. 
\end{lemma}

\noindent {\bf Proof:}  By squaring the integrand and applying independence, 
we get the following, 
$$\int (\frac1{N} \sum_{i=1}^{N}
({\mathcal X}_{E_i}-\gamma (E_i)))^2 d\gamma 
= \frac1{N^2} \sum_{i=1}^{N} \gamma (E_i) (1 - \gamma (E_i)) 
< \frac1{N^2} \sum_{i=1}^{N} \gamma (E_i) .$$
The Cauchy-Schwartz inequality implies 
\begin{align*}
| \frac1{N} \sum_{i=1}^{N} (\gamma (E\cap E_i)-\gamma (E)\gamma (E_i)) |
&= | \int_{E} (\frac1{N} \sum_{i=1}^{N}({\mathcal X}_{E_i}-\gamma (E_i)))d\gamma | \\ 
&< \frac1N \sqrt{\sum_{i=1}^{N} \gamma (E_i)}.
\end{align*}

Thus,
$$\frac{ | \sum_{i=1}^{N} 
(\gamma (E\cap E_i)-\gamma (E)\gamma (E_i)) | }
{\sum_{i=1}^{N} \gamma (E_i)} < 
\frac{\sqrt{\sum_{i=1}^{N} \gamma (E_i)}}{\sum_{i=1}^{N} 
\gamma (E_i)} \to 0$$
as $N\to \infty$, since $\sum_{i=1}^{\infty} 
\gamma (E_i)=\infty$. 
Therefore, the lemma is established for every $\epsilon > 0$ . 
$\ \ \ \ \Box$

\vskip .2in 
\noindent 
Now we are ready for the proof of our second main theorem. 

\begin{proof}[Proof of Theorem \ref{RatErg}] 
Let $\alpha, \beta \in \R$ be such that $0 < \alpha < 1$ and $\alpha \beta > 1$. 
Let $F$ be any set of positive finite measure. 
If we assume $T$ is $\beta$-rationally ergodic, then, by Lemma \ref{smallintegral}, 
there exist $\delta_0 > 0$ and $n_0 \in \mathbb N$ 
such that if $F^{\prime} \subset F$ 
satisfies $\mu(F^{\prime}) < \delta_0$, then for $t \geq n_0$, 
\[
\frac{\int_{F^{\prime}} \sum_{i=0}^{t-1} I_{F} (T^i x) d\mu}
{\int_{F} \sum_{i=0}^{t-1} I_{F} (T^i x) d\mu} < 1 . 
\]

Let $\delta = \min{ \{ \delta_0, {1}/{10} \} }$. 
Choose $N\in \mathbb N$ such that $N > \frac{1}{\delta}$ 
and there exists a union $J$ of intervals in $C_N$ such that 
\[
\frac{\mu (F \triangle J)}{\mu(J)} < 1 -  \sqrt{1 - \delta^2} .
\]
Let $\mu_N$ be normalized $\frac{\mu}{\mu(C_N)}$ probability measure on $C_N$. 
It is straightforward to see that the sets $C_N \cap C_n(k_n - 1)$ 
are independent for $n\geq N$ and 
$\sum_{n=N}^{\infty} \mu_N(C_N \cap C_n(k_n - 1)) = \infty$. 
Hence, by Lemma \ref{mixlem}, there exists $n > N$ such that 
\begin{align*}
\mu_N (F\cap J \cap C_N &\cap C_n(k_n - 1)) \\ 
&> \sqrt{1 - \delta^2} \mu_N (F\cap J) \mu_N (C_N \cap C_n(k_n - 1) ) \\ 
&> (1 - \delta^2) \mu_N(J) \mu_N(C_N \cap C_n(k_n - 1)) \\ 
&= (1 - \delta^2) \mu_N(J \cap C_N \cap C_n(k_n - 1)) 
\end{align*}
and such that both $H_n \geq n_0$ and 
\[
 \frac{ 2^{2\beta} \mu(F)^{\beta - 1} }
{ \frac{ \lfloor n^{\alpha} \rfloor^{\beta} }{(n + 1)} 
(1 - \delta)^2 ( 1 - 5\delta - \frac{2 \lfloor n^{\alpha} \rfloor}{m_{n}} ) } 
< \delta . 
\]
The set $J \cap C_N \cap C_n(k_n - 1)$ is a union of subintervals 
in the sub-tower $C_n(k_n - 1)$. 
Suppose $\bar{J} = J \cap C_N \cap C_n(k_n - 1) = \bigcup_{i=0}^{p - 1} J(i)$ 
where each $J(i)$ is a subinterval in $C_n(k_n - 1)$. 
Define 
\[
G = \{ J(i)\subset J : \mu_N (J(i)\cap F) \geq (1-\delta)\mu(J(i)) \} . 
\]
For convenience, associate $G = \bigcup_{J(i)\in G} J(i)$. 
Then $\mu_N(G) > (1 - \delta )\mu_N(\bar{J})$. 
If $q\in \mathbb N$ such that $0\leq q < \lfloor n^{\alpha}\rfloor$, then 
for $J_j, J_k \in G$, 
\[
\mu ((F\cap J_j) \cap (\bigcup_{i=0}^{h_{n}-1} T^{-qh_{n} - i} (F\cap J_k)) 
> (1 - 2 \delta - \frac{\lfloor n^{\alpha} \rfloor}{m_{n}} ) \mu(J_j) . 
\]
Thus, there exists a subset $J_j^* \subset J_j$ satisfying 
\[
\mu(J_j^*) > ( 1 - 4\delta - \frac{2 \lfloor n^{\alpha} \rfloor}{m_{n}} )\mu(J_j) 
\]
such that for $x\in J_j^*$, 
\[
\sum_{J_k\in G} \sum_{q=0}^{\lfloor n^{\alpha} \rfloor -1} \sum_{i=0}^{h_{n}-1} I_{F\cap J_k} (T^{qh_{n} + i} x) > \frac{ p \lfloor n^{\alpha} \rfloor }{2} . 
\]
Hence, 
\begin{align*}
\int_{F\cap J_j} ( \sum_{J_k\in G} \sum_{q=0}^{\lfloor n^{\alpha} \rfloor -1} \sum_{i=0}^{h_{n}-1} 
& I_{F\cap J_k} (T^{qh_{n} + i} x) )^{\beta} d\mu \\ 
&> (\frac{ p \lfloor n^{\alpha} \rfloor }{2})^{\beta} 
( 1 - 5\delta - \frac{2 \lfloor n^{\alpha} \rfloor}{m_{n}} ) \mu(J_j) 
\end{align*}
This implies 
\begin{align}
\int_F (\sum_{i=0}^{H_{n}-1} I_F(T^i x))^{\beta} & d\mu \geq 
\sum_{J_j\in G} \int_{F\cap J_j} ( \sum_{J_k\in G} \sum_{i=0}^{H_{n}-1} I_{F\cap J_k} (T^i x))^{\beta} d\mu \\ 
&> (\frac{ p \lfloor n^{\alpha} \rfloor }{2})^{\beta} 
( 1 - 5\delta - \frac{2 \lfloor n^{\alpha} \rfloor}{m_{n}} ) \mu(G) \\ 
&> (\frac{ p \lfloor n^{\alpha} \rfloor }{2})^{\beta} 
( 1 - 5\delta - \frac{2 \lfloor n^{\alpha} \rfloor}{m_{n}} ) (1 - \delta ) \frac{\mu(J)}{(n + 1)} \\ 
&> \frac{(1 - \delta)^2}{(n + 1)} (\frac{ p \lfloor n^{\alpha} \rfloor }{2})^{\beta} 
( 1 - 5\delta - \frac{2 \lfloor n^{\alpha} \rfloor}{m_{n}} ) \mu(F) . 
\end{align}

\noindent 
Let $\hat{J} = J\cap C_N \setminus C_n(k_n-1)$ and $\bigcup_{i=0}^{p-1} \hat{J}_i = \hat{J}$ 
where each $\hat{J}_i$ is a subinterval in $C_n\setminus C_n(k_n - 1)$. 
We have 
\begin{align}
\int_{F\cap \hat{J}} \sum_{i=0}^{H_{n}-1} I_F(T^i x) d\mu &\leq 
\int_{\hat{J}} \sum_{i=0}^{H_{n}-1} I_F(T^i x) d\mu \\ 
&= \sum_{j=0}^{p-1} \sum_{i=0}^{H_{n} - 1} \int_{T^{-i}\hat{J}_j} I_F(x) d\mu \\ 
&\leq \sum_{j=0}^{p-1} \mu (F) = p \mu(F) . 
\end{align} 
Since $\mu (F\setminus \hat{J}) < \delta \leq \delta_0$, then 
\[
\int_{F} \sum_{i=0}^{H_{n}-1} I_F(T^i x) d\mu \leq 2p \mu(F) . 
\]
Therefore, 
\begin{align*}
\frac{ ( \int_F \sum_{i=0}^{H_{n}-1} I_F(T^i x) d\mu )^{\beta} }
{ \int_F (\sum_{i=0}^{H_{n}-1} I_F(T^i x))^{\beta} d\mu } 
&< \frac{ 2^{\beta} p^{\beta} \mu(F)^{\beta} }
{ \frac{(1 - \delta)^2}{(n + 1)} (\frac{ p \lfloor n^{\alpha} \rfloor }{2})^{\beta} 
( 1 - 5\delta - \frac{2 \lfloor n^{\alpha} \rfloor}{m_{n}} ) \mu(F) } \\ 
&= \frac{ 2^{2\beta} \mu(F)^{\beta - 1} }
{ \frac{ \lfloor n^{\alpha} \rfloor^{\beta} }{(n + 1)} 
(1 - \delta)^2 ( 1 - 5\delta - \frac{2 \lfloor n^{\alpha} \rfloor}{m_{n}} ) } < \delta . 
\end{align*}
Since $\delta > 0$ may be chosen arbitrarily small, this contradicts 
the assumption that $T$ is $\beta$-rationally ergodic and completes 
the proof of our theorem. 
\end{proof}

%
%
%
%
%
%
%
%




\normalsize

\baselineskip=17pt

\bibliographystyle{amsalpha}
\bibliography{ErgodicBibMaster}

\end{document}